\documentclass[11pt]{amsart}
\usepackage{amsmath,amsthm,amsfonts,amssymb}
\usepackage{graphicx}

\usepackage[usenames,dvipsnames]{color}
\usepackage{hyperref}
\usepackage{soul}

\newenvironment{blue}{\color{blue}}{}
\newenvironment{green}{\color{green}}{}

\newcommand{\ov}[1]{{\overline{#1}}}

\title{Verbal subgroups of hyperbolic groups have infinite width}
\author{Alexei Myasnikov}
\address{Department of Mathematical Sciences\\Stevens Institute of Technology}
\curraddr{}
\email{amiasnikov@gmail.com}
\thanks{The first  author was partially supported by  NSF grants DMS-0914773 and DMS-1201550.}

\author{Andrey Nikolaev}
\address{Department of Mathematical Sciences\\Stevens Institute of Technology}
\curraddr{}
\email{andrey.nikolaev@gmail.com}

\newtheorem{cor}{Corollary}
\newtheorem{prop}{Proposition}
\newtheorem{theorem}{Theorem}
\newtheorem{lemma}{Lemma}

\newtheorem{problem}{Problem}

\begin{document}
\begin{abstract}
Let $G$ be a non-elementary  hyperbolic group. Let $w$ be a proper group word. We show that the width of the verbal subgroup  $w(G)=\langle w[G]\rangle$ is infinite. That is, there is no such  $l\in\mathbb Z$ that any $g\in w(G)$ can be represented as a product  of $\le l$ values of $w$ and their inverses.  As a consequence, we  obtain the same result for a wide class of relatively hyperbolic groups.
\end{abstract}

\maketitle

\section{Introduction}

In this paper we show that every non-elementary hyperbolic group $G$ is {\em  verbally parabolic}, i.e., every proper verbal subgroup of $G$ has infinite width.  This immediately implies that every group having a non-elementary hyperbolic image is also verbally parabolic, thus giving a wide class of verbally parabolic groups.

Let  $F(X)$ be a free group with  basis $X$ and  $W \subseteq F(X)$ a subset of $F(X)$. An element $g \in G$ is called a {\em $W$-element} if $g$ is the image in $G$ of some word $w \in W$ under a homomorphism $F(X) \to G$.  By $W[G]$ we denote the set of all $W$-elements in $G$ (the {\em verbal} set defined by $W$).  The set $W[G]$ generates the {\em verbal} subgroup $W(G)$. The verbal subgroups of groups were intensely studied in group theory especially with respect to relatively free groups and varieties of groups (see the book  \cite{Neumann} for details). If $W$ is finite then $W(G)$  is equal to $w(G)$ for a suitable single word $w$. From now on we always assume that $W$ is just a singleton $W = \{w\}$ and refer to the related verbal set and the subgroup  as  $w[G]$ and $w(G)$.

The {\em $w$-width} (or {\em $w$-length}) $l_w(g)$  of $g \in w(G)$ is  the least  natural number $n$ such that $g$ is a product of $n$   $w$-elements in $G$ or their inverses.  The {\em width} $l_w$ of $w$, as well as the {\em $w$-width} of the verbal subgroup $w(G)$, is defined by  $l_w = \sup\{l_w(g) \mid g \in w(G)\}$,  so it is either a natural number or $\infty$.

  Finite width of verbal subgroups plays an important part in several areas of groups theory:  finite groups, profinite and algebraic groups, polycyclic and abelian-by-nilpotent groups.

 In finite groups one of the earliest  questions on verbal width goes back to the Ore's paper \cite{Ore} where he asked if   the commutator length (the $[x,y]$-length) of every element in a non-abelian finite simple group is equal to 1 ({\em Ore Conjecture}). Only recently the conjecture was established by  Liebeck, O'Brian, Shalev and  Tiep \cite{LBST}. For other  recent spectacular results on verbal width in finite simple groups we refer to the papers \cite{LS}, \cite{Shalev}.

In profinite, as well as algebraic,  groups the interest to verbal width comes from a fact that verbal subgroups of finite width are closed in profinite, correspondingly Zariski,  topology. This shows up in   the proof of  Serre's result  \cite{Serre}: if $G$ is a finitely generated pro-p group then every subgroup of finite index is open.  Serre raised the question if the above result holds for arbitrary finitely generated profinite groups. In 2007 this conjecture was settled in affirmative by Nikolov and Segal \cite{NS}. Their proof is based in large on establishing uniform bounds on verbal width in finite groups. We refer to the book \cite{Segal} on verbal width in profinite and algebraic groups  and to a recent survey \cite{Nikolov2012}.

The study of verbal sets and subgroups in infinite groups was initiated by P. Hall \cite{H1,H2}. In 1960's several fundamental results on verbal width in infinite groups  were obtained in his school: Stroud proved (unpublished) that all finitely generated abelian-by-nilpotent groups $G$ are {\em verbally elliptic}, i.e., every verbal subgroup $w(G)$ of $G$ has finite $w$-width; meanwhile Rhemtulla showed in \cite{Rhe} that every not infinite dihedral free product $G = A \ast B$ of non-trivial groups $A$ and $B$ is verbally parabolic.  We refer again to the book \cite{Segal} for the history and detailed modern proofs.

Shortly afterwards, Hall's approach to verbal and marginal subgroups gave rise to an interesting  research in Malcev's algebraic school in Russia.  Firstly, Merzlyakov \cite{Mer} showed that linear algebraic groups are verbally elliptic, and then Romankov \cite{Romankov} (independently of Stroud and using different means) showed that finitely generated virtually  nilpotent  groups are all verbally elliptic; as well as  all virtually polycyclic groups (again, we refer to \cite{Segal} for concise proofs, generalizations and related new developments). Since Malcev's school was keen on model-theoretic problems in groups, it was  quickly recognized there that verbal subgroups of finite width are first-order definable in the group, so they provide a powerful tool in the study of elementary theories of verbally elliptic groups, in particular, nilpotent groups. Since then the elementary theories of finitely generated nilpotent groups and groups of $K$-points ($K$ is a field) of algebraic nilpotent groups were extensively studied  \cite{Mal,Ersh,MR1,MR2,MR3,M1,M2,M3}. Nowadays, it is known precisely which finitely generated nilpotent groups are elementarily equivalent \cite{Oger}; or what are arbitrary groups (perhaps, non-finitely generated) which are  elementarily equivalent to a given unipotent $K$-group \cite{M1,M2}, or a given unitriangular group $UT(n,\mathbb{Z})$ \cite{Bel}, or a  given finitely generated free nilpotent group  \cite{MS}.

In the opposite direction, Rhemtulla's remarkable proof of verbal parabolicity of non-abelian free groups and non-trivial free products gave rise to several interesting generalizations.  In \cite{Gr} Grigorchuk (see also Fujiwara~\cite{Fuji}) studied bounded cohomologies of group constructions, and proved  that verbal width of the commutator subgroup in a wide class of amalgamated free products and HNN extensions is infinite.  Bardakov showed that HNN extensions with proper associated subgroups and  one-relator groups with at least three generators are verbally parabolic \cite{Bar1}; as well as  braid groups $B_n$ \cite{Bar2}.  For amalgamated free products Dobrynina improved on the previous partial results by showing that $A\ast_U B$ is verbally parabolic, provided $U \neq A$ and $UbU \neq Ub^{-1}U$ for some $b \in B$ \cite{Dob}.

In this paper we show that many groups that have some kind of ``negative curvature'' are virtually parabolic.  The main result in its pure form (Theorem \ref{width}) states that  every non-elementary (i.e., non virtually cyclic)  hyperbolic group $G$ is verbally parabolic. As was mentioned above this  implies that every group that have a non-elementary  hyperbolic image is also verbally parabolic.

In particular, the following groups are verbally parabolic:  non-abelian  residually free groups, pure braid groups,  non-abelian right angled Artin groups --- all of them have free non-abelian quotients (so these are corollaries of the original Remtulla's result on free groups).
Less obvious examples include many relatively hyperbolic groups, e.g., non-elementary groups hyperbolic relative to proper residually finite subgroups \cite{Osi07}. Thus, the fundamental groups of complete finite volume manifolds of pinched negative curvature, $CAT(0)$ groups with isolated flats, groups acting freely on $\mathbb R^n$--trees are verbally parabolic.

The results above show that in the presence of negative curvature the verbal width is not a very sensitive characteristic of verbal subgroups. In this case  it seems more  convenient to consider a more smooth {\em stable } $w$-length,  which is defined for an element $g \in w(G)$ as the limit $\lim_{n\to \infty}\frac{l_w(g^n)}{n}$.
It is known that the stable commutator length relates to  an $L^1$  filling norm with rational coefficients, introduced by Gromov in \cite{Gr1}, see also Gersten's paper \cite{Ger}. In \cite{Gr2} Gromov studied stable commutator length and its relation with  bounded cohomology. We refer to a book \cite{Cal1}  by Calegary  on stable length of elements in groups.
 Interestingly,  Calegari showed in \cite{Cal2} that if a group $G$ satisfies a non-trivial law then the stable commutator length is equal to 0 for every element from $[G,G]$. So it seems the verbal width and stable length operate nicely in very different classes of groups.

 We formulate a few  open problems on verbal width in infinite groups.
 The first one, if true, would make proofs of many results easier.
 \begin{problem}
 Is it true that finite extensions of verbally parabolic groups are verbally parabolic?
 \end{problem}
 Notice that in general,  verbal parabolicity is not a geometric property. Indeed, there is a group $H$ with finite commutator width (the commutator width is equal to 1 in $H$), but some of its  finite extensions has infinite commutator width \cite[Exercise 3.2.2]{Segal}, so infinite width is not preserved by quasi-isometries.
The second problem  is an attempt to clarify if interesting groups without free subgroups can be verbally parabolic. It seems all known verbally parabolic groups have free non-abelian subgroups.
\begin{problem}
For which group words $w$ is $w$-width of Thompson group $F$ finite? Same question about Grigorchuk's group.
\end{problem}

We are grateful to Y.~Matsuda for his observation that, since the commutator subgroup of Thompson group $F$ is uniformly perfect~\cite{CFP}, the group $F$ has finite commutator width. It was shown in \cite{LMU} that the Grigorchuk group $\Gamma$ also has finite commutator width. The question remains open for other words.

In general, verbal sets $w[G]$ contain a lot of information about the group $G$. For example, the Membership Problem (MP) to the set $w[G]$ in $G$ is equivalent of solving in $G$ the homogeneous equations of the type $w(X) = g$, where  $g \in G$. Notice, that the Diophantine Problem (of solvability of arbitrary equations) is decidable in a free group $F$ \cite{Mak}, so the verbal sets are decidable.  However,   the actual algorithmic or formal language theoretic  complexity of verbal sets in $F$ is unknown.

\begin{problem}
Is it true that proper verbal subsets of a non-elementary hyperbolic group are non-rational?
\end{problem}
It is known that proper verbal set in free non-abelian groups are non-rational \cite{MRom}. Rhemtulla's technique from \cite{Rhe} played an important part in the proof of this result.

\section{Preliminaries}
\label{se:pre}

\subsection{Verbal sets and subgroups}
Let $F = F(X)$ be a free group with basis $X = \{x_1, \ldots,x_k, \ldots \}$, viewed as the set of reduced words in $X \cup X^{-1}$ with the standard multiplication. Throughout we use notation and definitions from Introduction.

By $w[G]^l$ we denote the set of elements $g\in G$ for which there exist $g_{ij}\in G$ such that
\[g=w^{\pm1}(g_{11},g_{12},\ldots)\cdot w^{\pm1}(g_{21},g_{22},\ldots)\cdots w^{\pm1}(g_{l1},g_{l2},\ldots).\]

A word $w = w(x_1, x_2, ..., x_n) \in F$ is termed  {\em proper} if there exist groups $G$ and $H$ such that $w[G] \neq  1$ and  $w[H] \neq H$. Notice that in such case $1 \neq w[G \times H]  \neq G\times H$, so the same group $G\times H$ satisfies both inequalities.

Any element $w  \in F$ can be written as  a product
\[w = x_1^{t_1}x_2^{t_2} ... x_n^{t_n} w',\]
 where $w'  \in [F,F]$, $t_1, \ldots, t_n \in \mathbb{Z}$. Since the exponents $t_1,  \ldots, t_n$ depend only on the element $w$, the number $e(w) = gcd(t_1, \ldots, t_n)$ is well-defined (here we put $e(w) = 0$ if $t_1 = \ldots = t_n = 0$).  If $e(w) = 0$ then we refer to $w$ as a {\em commutator word}.  A non-trivial commutator word is obviously proper (for instance, it is proper in $F$). If $e(w) > 0$ then  there exist integers $r_1, r_2, ..., r_n $ such that  $\sum_{i=1}^{n}r_it_i = e(w),$ so for an arbitrary group $G$ and an element $g \in G$ one has $w(g^{r_1}, g^{r_2}, ..., g^{r_n}) = g^{e(w)}.$ In particular, $e(w) = 1$ implies that  $w[G] = G$ for every group $G,$ so $w$ is not proper.  If $e(w) > 1$ then  $w$ is  proper, which can be seen in an infinite cyclic group. The argument above shows that a non-trivial  word  $w \in F $ is proper if and only if $e(w) \neq 1$.
The following is the main result of the paper.
\begin{theorem}\label{width}
Every non-elementary hyperbolic group $G$ is verbally parabolic, i.e., $w$-width of the verbal subgroup $w(G)$ is  infinite for each proper word $w$. 
\end{theorem}

\subsection{Hyperbolic geometry}
\label{se:hyperbolic-geometry}

In this section we recall some known facts about hyperbolic metric spaces (see \cite{Gr3,Eps,BH} for general references).

Let $\mathcal M$ be a metric space with distance $\mathop{\mathrm{dist}}(x,y)$ which we also denote by $|x-y|$. Given a path $p:I\to\mathcal M$, where $I\subseteq [0,\infty)$ is an interval, we can assume that $p$ is extended to a path $[0,\infty)\to\mathcal M$ by staying at endpoints. Fix a constant $K \in \mathbb{N}$.
Recall that two paths $p$ and $q$ in $\mathcal M$  {\em $K$-fellow travel} if $\mathop{\mathrm{dist}}(p(t), q(t))\le K$ for all $t\ge 0$ \cite{Eps}. They
{\em asynchronously $K$-fellow travel} if there are non-decreasing proper\footnote{A continuous map is {\em proper} if the preimage of every compact set is compact. As applied to a non-decreasing function $[0,\infty)\to[0,\infty)$, it means that the image is not bounded.} continuous functions
$\varphi, \psi : [0,\infty)\to [0,\infty)$ such that $\mathop{\mathrm{dist}}(p(\varphi(t)),q(\psi(t)))\le  K$ for all $t\ge 0$.

For a subset $Y \subseteq \mathcal M$ and a number $H \in \mathbb{N}$  by $\mathcal N_H(Y)$ we denote the closed $H$-neighborhood of $Y$. As usual, two subsets $Y_1$ and $Y_2$ of $\mathcal M$  are {\em $K$-close} if the Hausdorff distance between them does not exceed $K$, i.e.,
$Y_1\subseteq \mathcal N_K(Y_2)$ and $Y_2\subseteq \mathcal N_K(Y_1)$.

Let $I$ be an interval $I\subseteq [0,\infty)$. A path $p:I\to \mathcal M$ in a metric space is called $(\lambda,\varepsilon)$-quasigeodesic if
$$
\frac{1}{\lambda}\,|t-t'|-\varepsilon\le \mathop{\mathrm{dist}}(p(t),p(t'))\le \lambda \,|t-t'|+\varepsilon
$$
for all $t,t'\in I$.

Evidently, if paths $p$, $q$ asynchronously $K$-fellow travel then  they are $K$-close. The converse is true for quasigeodesic paths.

\begin{lemma}\label{fellow-equiv} Suppose $p, q$ are $K$-close $(\lambda,\varepsilon)$-quasigeodesics in a geodesic metric space $\mathcal H$ that originate from points $p_0,q_0$, respectively, such that $|p_0-q_0|\le K$, and terminate at points $p_1,q_1$, respectively, such that $|p_1-q_1|\le K$.
Then $p$ and $q$ asynchronously $K'$-fellow travel for some constant $K'=K'(\lambda, \varepsilon, K)$.
\end{lemma}
The above lemma seems to be well-known, but we were not able to find the original proof. The statement is essentially proved (albeit in different terms) in~\cite[Lemma~7.2.9]{Eps}, but for the sake completeness we provide a proof in the Appendix (Section~\ref{se:async}).

\begin{lemma}\label{fellow-0}
Let $\mathcal H$ be a $\delta$-hyperbolic geodesic metric space. Let $p$ be a geodesic path, and $q$ be a $(\lambda,\varepsilon)$-quasigeodesic
path in $\mathcal H$ joining points $P,Q$ and $P,S$, respectively. Suppose $H\ge 0$ is such that $|Q-S|\le H$. Then there exists $K=K(\delta,\lambda,\varepsilon,H)\ge 0$ such that $p,q$ asynchronously $K$-fellow travel.
\end{lemma}
\begin{proof}
It is  known (see, for example,~\cite[Section 7.2]{Gr3} or~\cite[Theorem III.1.7]{BH})  that such paths $p,q$ are $K$-close for some $K=K(\delta, \lambda, \varepsilon+H)$. Now, the asynchronous $K'$-fellow
travel property follows from Lemma~\ref{fellow-equiv}.
\end{proof}

\begin{lemma}\label{fellow}
Let $\mathcal H$ be a $\delta$-hyperbolic geodesic metric space. Let $p,q$ be two $(\lambda,\varepsilon)$-quasigeodesic
paths in $\mathcal H$ joining points $P_1,P_2$ and $Q_1,Q_2$, respectively. Suppose $H\ge 0$ is such that $|P_1-Q_1|\le H$ and
$|P_2-Q_2|\le H$. Then there exists $K=K(\delta,\lambda,\varepsilon,H)\ge 0$ such that $p,q$ asynchronously $K$-fellow travel.
\end{lemma}
\begin{proof} Join points $P_1$ and $Q_2$ with a geodesic $s$. Then by Lemma~\ref{fellow-0} $p,s$ and $s,q$ asynchronously $K$-fellow travel for some $K$, so they are $K$-close. Hence $p$ and $q$ are $2K$-close, and the result follows from Lemma \ref{fellow-equiv}.
\end{proof}

\subsection{Canonical representatives}\label{sub:reps}
In Section~\ref{sec:gaps} we use so-called canonical representatives for elements of hyperbolic groups. We follow the work of Dahmani and Guirardel~\cite{DG}, where they generalized Rips--Sela representatives for elements of torsion-free hyperbolic groups (see~\cite{RS}) to the case of arbitrary hyperbolic groups.
Below we give a brief account of the terminology and results essential to our present work, and we refer the reader to~\cite[Section 9]{DG} for details.

Let $\mathcal C=\mathcal C(G,A)$ be the Cayley graph of $G$ relative to the generating set $A$. Let $G$ be $\delta$-hyperbolic with respect to $A$ (assume without loss of generality $\delta\ge 1$). Let $V(\mathcal C)$ be the vertex set of the graph $\mathcal C$.
Consider so-called Rips complex $P_{50\delta}(\mathcal C)$ whose set of vertices is $V(\mathcal C)$, and whose simplices are subsets of $V(\mathcal C)$ of diameter at most $50\delta$ in the metric of $\mathcal C$ (see, for example,~\cite[Chapter III.$\Gamma$.3]{BH} for details regarding Rips complex). Let $\mathcal K$ denote the $1$-skeleton of its barycentric subdivision\footnote{
The {\em barycentric subdivision} of an $n$ vertex simplex $S$ consists of $n!$ simplices. Each of these simplices, with vertices $v_1,\ldots v_n$, corresponds to a permutation $p_1,\ldots,p_n$ of the vertices of $S$ so that $v_j$ is the barycenter of $p_1,\ldots,p_j$, $j=1,\ldots,n$.
}.
Note that there is an obvious embedding $\iota: V(\mathcal C)\to V(\mathcal K)$. We extend this map to paths in $\mathcal C$ by setting $\iota(v_1\overset{e}{\to}v_2)=\iota(v_1)\to \mathop{\mathrm{bary}}(e)\to \iota(v_2)$ for every edge $e$ and the corresponding barycenter $\mathop{\mathrm{bary}}(e)$ in $\mathcal K$. 

A path in a graph is called reduced if it has no backtracking subpaths consisting of two edges. Note that given a path, there is a unique corresponding reduced path obtained by removing all backtracks.
For a reduced path $q$ in $\mathcal K$ originating at $v\in\iota(V(\mathcal C))$ and terminating at $v'\in\iota(V(\mathcal C))$, by $\ov{q}$ we denote the group element $g^{-1}g'$, where $g,g'\in G$ correspond to $v$, $v'$, respectively. The set of reduced paths in $\mathcal K$ that originate at $\iota(1)$ and terminate in $\iota(V(\mathcal C))$ is denoted by $\mathcal Q$. If $q, q'\in \mathcal Q$, by the product $qq'\in \mathcal Q$ we mean the reduced path that corresponds to the concatenation $q\cdot \ov{q}q'$, where $\ov{q}q'$ is the translate of $q'$ by $\ov{q}\in G$. $q^{-1}\in\mathcal Q$ is defined accordingly.

Recall that a system $\mathcal E$ of equations in a group $G$ over a set of variables $Y$ is called {\em triangular} if every equation $\varepsilon\in\mathcal E$ has the form $y_1y_2y_3=1$, $y_i\in Y\cup Y^{-1}$.  The following claim was proved in~\cite[Proposition 9.10]{DG}. Informally speaking, it allows one to represent a system of triangular equalities by tripods in $\mathcal K$, rather than by triangles in the Cayley graph $\mathcal C$ (see Fig.~\ref{fig:reps}).
\begin{prop}\label{prop:reps}\cite{DG} Let $G$ be a hyperbolic group. Let $\mathcal C$, $\mathcal K$, $\mathcal Q$ be as above. 

Let $\mathcal{E}$ be a finite system of triangular equations over a finite set of variables $Y$, and $(g_y)_{y\in Y}$ a solution in $G$. Then for each variable $y\in Y\cup Y^{-1}$, there is a path $p_y\in\mathcal Q$ with $p_{y^{-1}}=(p_y)^{-1}$ and $\ov{p_y}=g_y$, and for each equation $\varepsilon\in \mathcal E$ written as $y_1y_2y_3=1$ and for each $i\in\{1,2,3\}$, there exist paths $s_{\varepsilon,i}\in\mathcal Q$ and $t_{\varepsilon,i}\in\mathcal Q$ such that
\begin{enumerate}
\item $p_{y_i}=s_{\varepsilon,i}t_{\varepsilon,i}s_{\varepsilon,i+1}^{-1}$ (here $i+1$ is taken modulo $3$),
\item $|t_{\varepsilon,i}|\le D_1$ and $\ov{t_{\varepsilon,1}t_{\varepsilon,2}t_{\varepsilon,3}}=1$ in $G$, where the constant $D_1$ depends only on the presentation of the group $G$ and the system $\mathcal E$.
\end{enumerate}
Moreover, paths $p_y$ and $s_{\varepsilon,i}$ may be assumed to be $(\lambda_1,\mu_1)$-quasigeodesic, where the constants $\lambda_1,\mu_1$ depend only on the presentation of the group~$G$.
\end{prop}
\begin{figure}[h]
 \centering
 \includegraphics[height=1.5in]{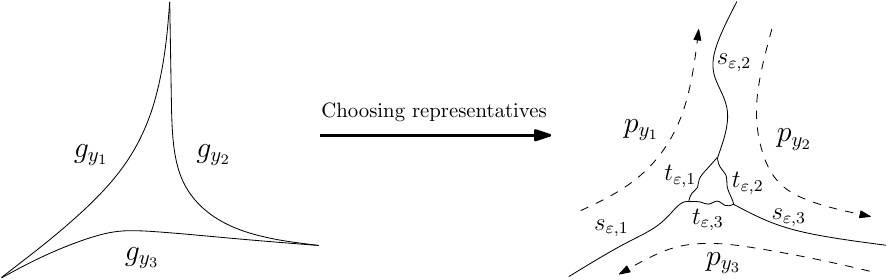}
 \caption{For each variable $y\in Y$, a path $p_y$ is chosen so that all triangles arising from equations $\varepsilon\in\mathcal E$ degenerate into tripods.}\label{fig:reps}
\end{figure}

The collection of paths $(p_y)_{y\in Y}$ above is called a system of {\em $D_1$-canonical representatives} for $\mathcal E$ and $(g_y)_{y\in Y}$.

\section{Big Powers Products}
The following notation is in use throughout the whole paper. Virtually cyclic groups are called elementary. Let $G$ be a non-elementary hyperbolic group with a  generating set $A$.  Put $\mathcal A=A\cup A^{-1}$ and denote by ${\mathcal A}^\ast$ the free monoid generated by $\mathcal A$.
For a word $h \in {\mathcal A}^\ast$  by $\ov h$ we denote the image of $h$ under the canonical epimorphism  ${\mathcal A}^\ast \to G$.  We say that a word $h$ is (quasi)geodesic in $G$ if the corresponding
path in the Cayley graph $\mathcal C(G,A)$ of $G$ relative to $\mathcal A$ is (quasi)geodesic.  For $g\in G$, by $|g|$ we denote the geodesic length of $g$. Let $\delta$ be a hyperbolicity constant of $G$ relative to $A$.

For an element $g\in G$ of infinite order, $E(g)$ denotes the unique maximal elementary subgroup of $G$ containing $g$.

The following lemma is due to Ol'shanski~\cite{Ol}.
\begin{lemma}\label{olsh} \cite[Lemma~2.3]{Ol} Let $h_1,\ldots,h_\ell$ be words in $\mathcal A$ representing elements
$\ov h_1,\ldots,\ov h_\ell$ of infinite order in $G$ for which $E(\ov h_i)\neq
E(\ov h_j)$ for $i\neq j$.
Then there exist constants $\lambda=\lambda(h_1,\ldots,h_\ell)>0$,
 $\varepsilon=\varepsilon(h_1,\ldots,h_\ell)\ge 0$, and $N=N(h_1,\ldots,h_\ell)>0$ such that
any word of the form
\[
h_{i_1}^{m_1}h_{i_2}^{m_2}\cdots h_{i_s}^{m_s},
\]
where  $i_k\in \{1,\ldots,\ell\}$,  $i_k\neq i_{k+1}$ for $k=1,\ldots, s-1$, and $|m_k|> N$ for $k=2,\ldots, s-1$,  is $(\lambda,\varepsilon)$-quasigeodesic in $G$.
\end{lemma}

Elements $g_1,g_2\in G$ 
are termed {\em commensurable} if $h^{-1}g_1^{n_1}h=g_2^{n_2}$ for some $h\in G$ and $n_1,n_2\in \mathbb Z$, 
and $n_1,n_2$ not both zero. Note that within this terminology, an element of finite order is commensurable with any other group element. This is a convenience choice, and will not be of much consequence. Observe also that $E(g_1)\neq E(g_2)$ for non-commensurable $g_1$ and $g_2$.
Further, we say that an element $g\in G$ is of {\em dihedral type} if $h^{-1}g^{n_1}h=g^{n_2}$ for some $h\in G$ and $n_1,n_2\in \mathbb Z$, $n_1n_2\le 0$.

\begin{prop}
\label{C1} Let $G$ be a non-elementary hyperbolic group.
Let $h,h_0,h_1$ be words in $\mathcal A$ such that elements $\ov h,\ov h_0,\ov h_1\in G$ are of infinite order and pairwise non-commensurable. Then
there exists a number $L>0$ and constants $\lambda, \varepsilon$ that depend on $h, h_0,h_1$,
such that the following condition holds. Let $g$ be a word of the form
\[
g=g_1^{m_1}g_2^{m_2}\cdots g_k^{m_k},
\]
where $g_i\in \{h^{\pm 1},h_0^{\pm 1},h_1^{\pm 1}\}$, $m_i\ge L$ ($i=1,\ldots, k$), and $g_{i}\neq g_{i+1}^{\pm 1}$ ($i=1,\ldots, k-1$).
Then $g$ is $(\lambda, \varepsilon)$-quasigeodesic in $G$.
\end{prop}
\begin{proof}
The statement follows from Lemma~\ref{olsh} since the elements $\ov h,\ov h_0,\ov h_1$ are pairwise non-commensurable and therefore $E(\ov h), E(\ov h_1), E(\ov h_2)$ are pairwise distinct.
\end{proof}

The following corollary replaces $h,h_0,h_1$ with $b=h^{L},f_0=h_0^{L},h_1=f_1^L$.

\begin{cor}
\label{C1''} Let $G$ be a non-elementary hyperbolic group.
Let $h,h_0,h_1$ be words in $\mathcal A$ such that elements $\ov h,\ov h_0,\ov h_1\in G$ are of infinite order and pairwise non-commensurable. Let $L,\lambda,\varepsilon$ be as provided in Proposition~\ref{C1}. Put $b=h^L,f_0=h_0^L,f_1=h_1^L$. Let $g$ be a word of the form
\[
g=g_1^{m_1}g_2^{m_2}\cdots g_k^{m_k}
\]
where $g_i\in D=\{b^{\pm 1},f_0^{\pm 1},f_1^{\pm 1}\}$, $m_i>0$ ($i=1,\ldots, k$), and $g_{i}\neq g_{i+1}^{\pm 1}$
($i=1,\ldots, k-1$).
Then $g$ is $(\lambda, \varepsilon)$-quasigeodesic in $G$.
\end{cor}

\begin{cor}
\label{C1'} Let $G$ be a non-elementary hyperbolic group.
Let words $b,f_0,f_1$ be as in Corollary~\ref{C1''}. Suppose
\[
\ov g_1^{m_1}\ov g_2^{m_2}\cdots \ov g_k^{m_k}=\ov{g'_1}^{m'_1}\ov{g'_2}^{m'_2}\cdots \ov{g'_l}^{m'_l},
\]
where $g_i,g'_i\in D=\{b^{\pm 1},f_0^{\pm 1},f_1^{\pm 1}\}$, $m_i,m'_i>0$, and
$g_{i}\neq g_{i+1}^{\pm 1}$, $g'_{i}\neq (g'_{i+1})^{\pm 1}$. Then $k=l$ and $g_i=g'_i$.
\end{cor}
\begin{proof}
Apply Corollary~\ref{C1''} to the product
\[(g_k^{-1})^{m_k}\cdots (g_2^{-1})^{m_2}(g_1^{-1})^{m_1}{g'_1}^{m'_1}{g'_2}^{m'_2}\cdots {g'_l}^{m'_l}.\]
\end{proof}

The following lemma is well known but we exhibit it here for convenience.
\begin{lemma}\label{le:comm}Let $G$ be a group generated by a finite set $X$. Let $g_1,g_2$ be two words in the alphabet $X\cup X^{-1}$. Let $p$ and $q$ be two paths in the Cayley graph of $G$ relative to $X$. If a segment of $p$ labeled by $g_1^m$, with $m\ge (2|X|)^{1+K+|g_2|}$, asynchronously $K$-fellow travels with a segment of $q$ labeled by $g'g_2^{m'}g''$, with $|g'|,|g''|<|g_2|$, $m'\in\mathbb Z$, then $\ov{g_1}$ and $\ov{g_2}$ are commensurable in $G$. Moreover, in that case the conjugated powers of $\overline{g}_1, \overline{g}_2$ can be chosen to be non-negative.
\end{lemma}
\begin{proof}
Note that in such a case, the endpoint of each copy of $g_1$ in $p$ is connected by a path of length at most $K$ with a point on $q$. Therefore, the endpoint of each copy of $g_1$ is connected by a path labeled by a word $u_i$, $i=1,\ldots,m$ with an endpoint of a copy of $g_2$ so that $|u_i|\le K+|g_2|$.
Since there are at most $(2|X|)^{1+K+|g_2|}$ words of length $\le K+|g_2|$, there is at least one repetition $u_i=u_j$, $j>i$, yielding that in $G$ one has
\begin{equation}
\ov{g}_1^{j-i}=\overline{u}_i \ov{g}_2^{k} \overline{u}_i^{-1},\label{eq:conjugated}
\end{equation} i.e. that $\ov g_1$, $\ov g_2$ are commensurable (see Figure~\ref{fig:comm2}).
\begin{figure}[h]
 \centering
 \includegraphics[height=0.9in]{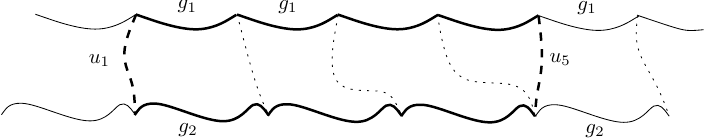}
 \caption{In this example, if $u_1=u_5$, then $\ov{g}_1^{4}=\overline{u}_1 \ov{g}_2^{3} \overline{u}_1^{-1}$.}\label{fig:comm2}
\end{figure}
\end{proof}

\begin{prop}\label{factor lemma}Let $G$ be a non-elementary hyperbolic group.
Let $b,f_0,f_1$ be words in $\mathcal A$ such that elements $\ov h,\ov h_0,\ov h_1\in G$ are of infinite order, non-dihedral type, and pairwise non-commensurable. Let constants $\lambda, \varepsilon$ be as provided by Proposition~\ref{C1}. Let $K\in\mathbb Z$,
$K>0$ be given. Then
there exists $M\in\mathbb Z$, $M>0$ (that depends on $b, f_0, f_1,\lambda,\varepsilon, K$) such that the following condition holds.
Let $g\in G$ be such that
\[
g=\ov c_1\ov g_1^{m_1}\ov g_2^{m_2}\cdots \ov g_k^{m_k}\ov c_2=\ov c_3\ov {g'}_1^{m'_1}\ov {g'}_2^{m'_2}\cdots \ov{g'}_l^{m'_l}\ov c_4
\]
where $|c_i|\le K$, $g_i,g'_i\in D=\{b^{\pm 1},f_0^{\pm 1},f_1^{\pm 1}\}$, $m_i\ge M$ for $i=1,2,\ldots, k$, $m'_i\ge M$ for $i=2,\ldots,l-1$,
and $g_{i-1}\neq g_{i}^{\pm 1}$,
$g'_{i-1}\neq {g'_{i}}^{\pm 1}$.
Then one of the following occurs:
\begin{itemize}
\item[(a)] $(g_1,g_2,\ldots,g_{k})$=$(g'_2,g'_3,\ldots,g'_{l-1})$, or
\item[(b)] $(g_1,g_2,\ldots,g_{k})$=$(g'_2,g'_3,\ldots,g'_{l})$, or
\item[(c)] $(g_1,g_2,\ldots,g_{k})$=$(g'_1,g'_3,\ldots,g'_{l-1})$, or
\item[(d)] $(g_1,g_2,\ldots,g_{k})$=$(g'_1,g'_3,\ldots,g'_{l})$.
\end{itemize}
In particular, $(g_1,g_2,\ldots,g_{k})$ is a subsequence of $(g'_1,g'_2,\ldots,g'_{l})$.
\end{prop}
\begin{proof}
By Lemma~\ref{olsh}, $g_1^{m_1}\cdots g_k^{m_k}$ and ${g'_1}^{m'_1}\cdots {g'_l}^{m'_l}$
are $(\lambda,\varepsilon)$-quasigeodesics. Then $c_1g_1^{m_1}\cdots g_k^{m_k}c_2$ and $c_3{g'_1}^{m'_1}\cdots {g'_l}^{m'_l}c_4$
are $(\lambda,\varepsilon_0)$-quasigeodesics, where $\varepsilon_0$ depends of $\lambda, \varepsilon, K$. By Lemma~\ref{fellow}, paths corresponding to these two words asynchronously $H$-fellow travel for some $H=H(\delta,\lambda,\varepsilon_0,K)$.

Put 
\[
N=(2|X|)^{\max\{|b|,|f_0|,|f_1|\}+H+1}, \mbox{ and } M>2N.
\]
For each $j$, $1\le j\le k$, consider the points $o'_j$ and $t'_j$ of ${g'_1}^{m'_1}\cdots {g'_l}^{m'_l}$ that correspond to endpoints of $g_j^{m_j}$ under the aforementioned fellow travel. Let $o'_j$ and $t'_j$ belong to some ${g'_\alpha}^{m'_\alpha}$ and ${g'_\beta}^{m'_\beta}$, respectively, with $\alpha\le\beta$. If $\beta-\alpha\le 1$, then by the choice of $M$, there is a subword of $g_j^{m_j}$ labeled by $g_j^N$ that fellow travels with a subword of ${g'_\alpha}^{m'_\alpha}$ or ${g'_\beta}^{m'_\beta}$. By Lemma~\ref{le:comm}, a positive power of $g_j$ is conjugate to a positive power of  $g'_\alpha$ or $g'_\beta$. If $\beta-\alpha\ge 2$, then ${g'_{\alpha+1}}^{m'_{\alpha+1}}$ fellow travels with a subword of $g_j^{m_j}$. Since $1<\alpha+1<l$, $m'_{\alpha+1}\ge M$, so by Lemma~\ref{le:comm}, a positive power of $g'_{\alpha+1}$ is conjugate to a positive power of $g_j$. In either case, a positive power of $g_j$ is conjugate to a positive power of $g'_{i_j}$ with $\alpha\le i_j\le \beta$. By the choice of $b,f_1,f_2$, this implies that $g_j=g'_{i_j}$.

The above shows that $(g_1,g_2,\ldots, g_k)$ is a subsequence of $(g'_1,g'_2,\ldots,g'_l)$. By the same argument, $(g'_2,\ldots, g'_{l-1})$ is a subsequence of $(g_1,\ldots,g_k)$. Together these two facts immediately imply the statement.
\end{proof}

By~\cite[Lemma~3.8]{Ol}, in a non-elementary hyperbolic group $G$ there exist elements $\ov h,\ov h_0,\ov h_1\in G$ of infinite order that are of non-dihedral type and pairwise non-commensurable.
Let $L$ be as provided by 
Proposition~\ref{C1}. Put $b=h^{L},f_0=h_0^{L},h_1=f_1^L$. For the rest of the paper, fix the elements $b,f_0,f_1$ and the set $D=\{b^{\pm 1},f_0^{\pm 1},f_1^{\pm 1}\}$. Let $M\in\mathbb Z$ be as described in Proposition~\ref{factor lemma}. Then we say that
$(b,f_0,f_1,M)$ is a {\em big powers tuple}. Further, an expression of the form
\begin{equation}\label{Rhem}
g_1^{m_1}g_2^{m_2}\cdots g_k^{m_k},\quad \mbox{where}\ g_i\in D,\ m_i\ge M,\ g_{i}\neq g_{i+1}^{\pm 1}
\end{equation}
is termed a $(b,f_0,f_1,M)$ {\em big powers product}. A path in the Cayley graph corresponding to a big powers product is called a {\em big powers path}.

Define a set of elements $R=R(b,f_0,f_1,M)\subseteq G$ to be
\begin{align*}
R = \{&g\in G\mid \exists k\in\mathbb N,g_i\in D, m_i\ge M,\ g_{i}\neq g_{i+1}^{\pm 1},\\
&g=\ov g_1^{m_1}\ov g_2^{m_2}\cdots \ov g_k^{m_k}\},
\end{align*}
that is the set of elements representable in the form~(\ref{Rhem}).

{\sc Remark.} Note that it is false in general that if $xy\in R$, then ${x\in R}, {y\in R}$.

\section{Counting gaps}\label{sec:gaps}
For a big powers product $g=g_1^{m_1}\cdots g_k^{m_k}$, $g_i\in D$, $m_i\ge M$, its subproduct of the form
\[
 g_\mu^{m_{\mu}}  g_{\mu+1}^{m_{\mu+1} }\cdots  g_{\mu+\nu}^{m_{\mu+\nu}}  g_{\mu+\nu+1}^{m_{\mu+\nu+1}},\quad g_\mu=g_{\mu+\nu+1}=b,
\]
where $g_i\neq b,b^{-1}$ for $i=\mu+1,\ldots,\mu+\nu$, is called a {\em $b$-syllable}.
Similarly, a subproduct of the form
\[
 g_\mu^{m_{\mu}}  g_{\mu+1}^{m_{\mu+1} }\cdots  g_{\mu+\nu}^{m_{\mu+\nu}}  g_{\mu+\nu+1}^{m_{\mu+\nu+1}},\quad g_\mu=g_{\mu+\nu+1}=b^{-1},
\]
where $g_i\neq b,b^{-1}$ for $i=\mu+1,\ldots,\mu+\nu$, is called a $b^{-1}$-syllable. We say $b^{\pm 1}$-syllable as a shortcut for a $b$- or $b^{-1}$-syllable. \textsl{}With each $b$-syllable $s$ we associate an integer $\omega_s\in\mathbb Z$
so that
\[
\omega_s=\varepsilon_1+\varepsilon_2+\cdots+\varepsilon_\nu,
\] where $\varepsilon_i=0$ if $g_{\mu+i}=f_{1}^{\pm 1}$, and $\varepsilon_i$ is such that $g_{\mu+i}=f_{0}^{\varepsilon_i}$, otherwise. The number
$\omega_s$ is called the {\em $b$-gap} associated with $s$. We define $b^{-1}$-gaps in the same way. Following Rhemtulla~\cite{Rhe},
given $\omega\in \mathbb Z$ and a big powers product~$x$ we define $\delta_\omega(x)$ to be the number of $b$-gaps $\omega$ in $x$.
By $\delta^*_\omega(x)$ we denote the number of $b^{-1}$-gaps $\omega$. Observe that
\begin{equation}\label{eq:delta_star}
\delta_\omega(x)=\delta^*_{-\omega}(x^{-1}).
\end{equation}
Further, by $\gamma_d(x)$, $d\in\mathbb Z$, we denote the number of integers $\omega\in\mathbb Z$ such that
\[
\delta_\omega(x)-\delta^*_{-\omega}(x)\neq 0\ \mod d
\]
(if $d=0$ this means $\delta_\omega(x)-\delta^*_{-\omega}(x)\neq 0$ as integers). 
Note that
\[
\gamma_d(x)=\gamma_d(x^{-1}).
\]
Denote $\delta_\omega(x)-\delta^*_{-\omega}(x)$ by $\Delta_\omega(x)$. For a big powers path $p$, by $\delta_\omega(p)$ (or $\delta^*_\omega(p)$, or $\Delta_\omega(p)$) we mean the value of $\delta_\omega$ (respectively, $\delta^*_\omega$, $\Delta_\omega$) for the corresponding word. Further, note that by Corollary~\ref{C1'}, the collection of $b$- and $b^{-1}$-syllables is the same for any big powers product representing the element $g\in R$, which means that $\delta_\omega,\delta^*_\omega, \Delta_\omega, \gamma_d$ are also well defined as functions of $g\in R$.

We will show that for each $w\in F$, $d=e(w)$ and for each $l$ there is an appropriate choice of $b,f_0,f_1,M$ such that
$\gamma_d$ is bounded on $w[G]^l\cap R$ and unbounded on $w(G)\cap R$. In what follows the word $w$ and therefore $d=e(w)$ are fixed, so we omit the subscript $d$ and write $\gamma(x)$ instead of $\gamma_d(x)$.

To this end, we use a ``coarse'' version $\widetilde{\Delta_\omega}$ of $\Delta_\omega$, defined as follows. 
Let $\mathcal C=\mathcal C(G,A)$ be the Cayley graph of $G$ relative to the generating set $A$.
Let $\mathcal K$ be as described in Section~\ref{sub:reps}, i.e. the $1$-skeleton of barycentric subdivision of Rips complex $P_{50\delta}(\mathcal C)$, and let $\iota$ be the embedding $\iota: V(\mathcal C)\to V(\mathcal K)$ extended to paths in $\mathcal C$.

Let $q$ be an arbitrary path in $\mathcal K$. Let $H>0$ be a parameter to be chosen later. Denote
\begin{equation}\label{eq:defn_delta_tilde}
\widetilde{\delta_\omega}(q)=\max \sum_i \delta_\omega(p_i),
\end{equation}
where the maximum is taken over all systems of  non-intersecting subpaths $q_1,q_2,\ldots$ of $q$ that asynchronously  $H$-fellow travel with $\iota(p_1),\iota(p_2),\ldots$, respectively, where $p_1,p_2,\ldots$ are big power paths in $\mathcal C$. We define $\widetilde{\delta^*_\omega}$ accordingly and put
\begin{equation}\label{eq:delta-hat}\widetilde{\Delta_\omega}(p)=\widetilde{\delta_\omega}(p)-\widetilde{\delta^*_{-\omega}}(p).
\end{equation}
Note that $\widetilde{\Delta_\omega}(p)=-\widetilde{\Delta_{\omega}}(p^{-1})$.

\begin{lemma}\label{le:delta_vs_tilde} Let $\varepsilon_0,\lambda_0$ be given. Let $q$ be a $(\varepsilon_0,\lambda_0)$-quasigeodesic path in $\mathcal K$. Suppose there is a big powers path $p$ in $\mathcal C$ such that $q$ and $\iota(p)$ asynchronously $H$-fellow travel. Then if $M$ in~\eqref{Rhem} is large enough, $\widetilde{\Delta_\omega}(q)=\Delta_\omega (p)$.
\end{lemma}
\begin{proof}
Indeed, in such event, $\widetilde{\delta_\omega}(q)\ge \delta_\omega(p)$ immediately from~\eqref{eq:defn_delta_tilde}. To show the reverse inequality, we consider a subpath $q_1$ of $q$ that asynchronously $H$-fellow travels with a path $\iota(p_1)$, where $p_1$ is labeled by a big powers product $s_1$, as shown in Fig.~\ref{fi:prop3}. Since we are interested in counting $b$-gaps, we may assume that $s_1$ starts and ends with a power of $b$. Let $q_1$ fellow travel with a subpath $\iota(p_1')$ of $\iota(p)$, where  $p$ is labeled by a big powers product $g_1^{m_1}\cdots g_k^{m_k}$. Considering $(H+2\max\{|b|,|f_0|,|f_1|\})$- instead of $H$-fellow travel, we may assume that $p_1'$ is labeled by a word of the form $g_j^{m'_j}g_{j+1}^{m_{j+1}}\cdots g_{l-1}^{m_{l-1}}g_{l}^{m'_l}$, where $m'_j\le m_j$ and $m'_l\le m_l$. By fellow travel, the initial point of $q_1$ is connected with initial points $P,P'$ of paths $\iota(p_1)$, $\iota(p_1')$ by paths $r_1,r_1'$ in $\mathcal K$, respectively. Note that the endpoints $P',P$ of the path ${r_1'}^{-1}r_1$ are in $\iota(V(\mathcal C))$, so there is a word $c_1$, $|c_1|\le 2(H+2\max\{|b|,|f_0|,|f_1|\})$, read as a label of a path in $\mathcal C$ from $P'$ to $P$. Similarly, there is word $c_2$, $|c_2|\le 2(H+2\max\{|b|,|f_0|,|f_1|\})$, read as a label of a path in $\mathcal C$ from the terminal endpoint of $p_1$ to that of $p_1'$. (See Fig.~\ref{fi:prop3}.) This provides the equality
\begin{figure}[h]
 \includegraphics[height=1.6in]{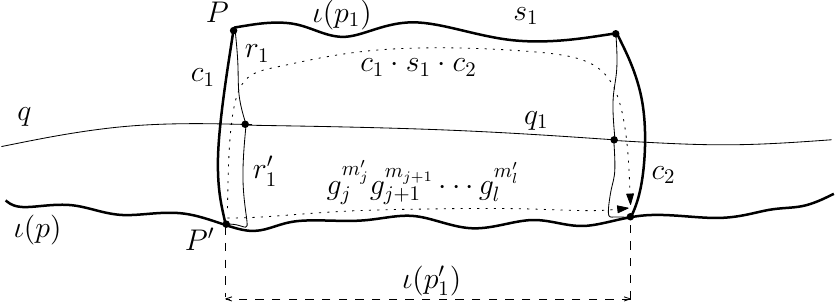}
 \caption{Fellow travel of $q_1$ with $\iota(p_1)$ and $\iota(p_1')$ provides equality~\eqref{eq:use_prop3}. Paths in $\mathcal K$ that are in $\iota(\mathcal{C})$ are represented by thicker lines.}\label{fi:prop3}
\end{figure}
\begin{equation}\label{eq:use_prop3}
\ov{c_1}\cdot \ov{s_1}\cdot \ov{c_2}
=
\ov{g_j}^{m'_j}\ov{g_{j+1}}^{m_{j+1}}\cdots \ov{g_{l-1}}^{m_{l-1}}\ov{g_{l}}^{m'_l},
\end{equation}
where $|c_i|\le 2(H+2\max\{|b|,|f_0|,|f_1|\})$, $i=1,2$. Applying Proposition~\ref{factor lemma} to this equality, we conclude that if $M$ in~\eqref{Rhem} is large enough, there is a big powers product $s_1'$ within the product $g_j^{m_j}g_{j+1}^{m_{j+1}}\cdots g_{l-1}^{m_{l-1}}g_{l}^{m_l}$ with the same values of $b$-gaps as $s_1$.

Now note that by quasigeodesity of $q$ and $p$, the subpaths $p_1', p_2'$ of $p$ corresponding as described above to non-overlapping paths $q_1,q_2$ of $q$ cannot overlap on a subpath longer than some fixed number (depending on $\varepsilon_0,\lambda_0,\varepsilon, \lambda$), which is enough to guarantee that the respective big powers products $s_1'$ and $s_2'$ found within labels of $p_1',p_2'$ do not overlap more than on a power $b^{M_0}$, where $M_0$ depends on $b,$ $\varepsilon_0,\lambda_0,\varepsilon, \lambda$. Assume $M>M_0$.

It follows that for each system of paths $q_i$ in~\eqref{eq:defn_delta_tilde},
\[
\delta_\omega(p)\ge \sum_i\delta_\omega(p_i),
\]
so $\delta_\omega(p)\ge \widetilde{\delta_\omega}(p)$, i.e. $\delta_\omega(p)= \widetilde{\delta_\omega}(p)$. By~\eqref{eq:delta_star} it follows that the same holds for $\delta^*_{-\omega}$, resulting in $\widetilde{\Delta_\omega}(q)=\Delta_\omega (p)$.
\end{proof}
\medskip

\begin{lemma}\label{le:product}
Let $q'$, $q''$ be reduced paths in $\mathcal K$ such that the concatenation $q'q''$ is defined. Then if $M$ in~\eqref{Rhem} is large enough, $\widetilde{\Delta_\omega}(q'q'')=\widetilde{\Delta_\omega}(q')+\widetilde{\Delta_\omega}(q'')$ for all $\omega$ except, perhaps, nine values.
\end{lemma}
\begin{proof} Let $M_0$ be as provided by Proposition~\ref{factor lemma} with $K=2(H+2\max\{|b|,$ $|f_0|,$ $|f_1|\})$. Then by triangle inequality, definition of a quasigeodesic, and that of asynchronous fellow travel, there is $M\ge M_0$ such that if a big powers path labeled by $b^m$ asynchronously $(H+2\max\{|b|,|f_0|,|f_1|\})$-fellow travels with a big powers path labeled by $g^{M}$, $g\in\{b^{\pm 1}, f_0^{\pm 1}, f_1^{\pm 1}\}$, then $m\ge M_0$. Below we show that such $M$ suffices.

Let $q$ be a reduced path in $\mathcal K$ and $r_i$, $i=1,2,3$, its three subpaths that share a common point and asynchronously $H$-fellow travel with big power paths $\iota(p_i)$ labeled by $b$- or $b^{-1}$-syllables $s_i$, $i=1,2,3$. Without loss of generality we assume that the initial endpoint of the path $r_3$ is the leftmost (the earliest on $q$) of the three initial endpoints of paths $r_i$; and that the terminal endpoint of $r_2$ is the rightmost of the two terminal endpoints of paths $r_1,r_2$. Thus, the path $r_1$ is covered by $r_2, r_3$.

\medskip
{\sc Claim. } The $b$-syllable $s_1$ has the same $b$-gap $\omega_{s_1}$ as $s_2$ or $s_3$, provided $|\omega_{s_i}|\ge 3$, $i=1,2,3$.

\medskip
Indeed, let $s_1=g_1^{m_1}g_2^{m_2}\cdots g_{k-1}^{m_{k-1}}g_k^{m_k}$, $g_1=g_k=b$. This splits $r_1$ into subpaths $r_1=r_1^{(1)}\cdots r_1^{(k)}$ such that each $r_1^{(j)}$, $j=1,\ldots,k$, fellow travels with a subpath of $\iota(p_1)$ labeled by $g_j^{m_j}$. Represent $r$ as $r=r'r''$, where, for definiteness, $r'$ is contained in $r_3$, and $r''$ is the maximal subpath of $r$ contained in $r_2$ (without loss of generality, $r''$ is non-empty), as shown in Fig.~\ref{fi:split_r}.
\begin{figure}[h]
 \includegraphics[height=1.2in]{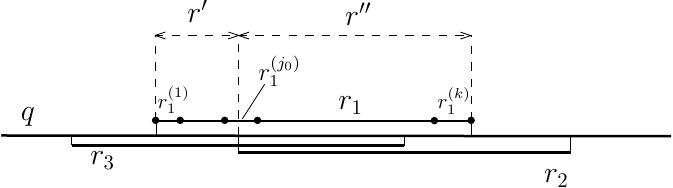}
 \caption{$r_1=r'r''$. The initial endpoint of $r''$ is contained in $r_1^{(j_0)}$.}\label{fi:split_r}
\end{figure}
Let the initial endpoint of $r''$ be contained in $r_1^{(j_0)}$ (if $r'$ is empty, put $j_0=1$).
Assume without loss of generality that $j_0\le (k+1)/2$.
Note for later use that since $|\omega_{s_1}|\ge 3$, $k\ge 5$, so $j_0\le (k+1)/2\le k-2$. Let $s_2={g'_1}^{m'_1}{g'_2}^{m'_2}\cdots {g'_l}^{m'_l}$, $g'_1=g'_l=b$ or $g'_1=g'_l=b^{-1}$. Then $r''$ asynchronously $(H+2\max\{|b|,|f_0|,|f_1|\})$-fellow travels with both a terminal subpath of $\iota(p_1)$ labeled by $g_{j_0}^{m}g_{j_0+1}^{m_{j_0+1}}\cdots g_k^{m_k}$ ($m\le m_{j_0}$) and an initial subpath of $\iota(p_2)$ labeled by ${g'}_1^{m'_1}\cdots {g'}_{k_0-1}^{m'_{k_0-1}}{g'}_{k_0}^{m'}$, $m'\le m'_{k_0}$. If $k_0>1$, we disregard the part of $r''$ corresponding to ${g'}_{k_0}^{m'}$. As shown in Fig.~\ref{fi:prop3_2},
\begin{figure}[h]
 \includegraphics[width=4.5in]{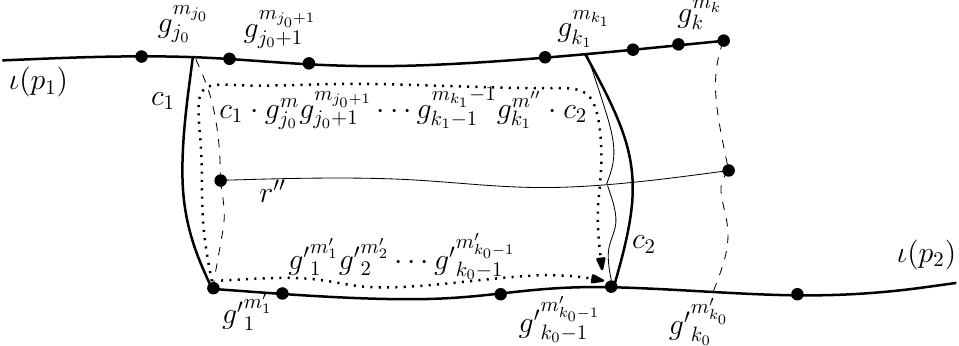}
 \caption{If $k_0>1$, the fellow travel of $r''$ with a subpath of $\iota(p_1)$ and a subpath of $\iota(p_2)$ delivers equality~\eqref{eq:le_product_1}. Paths in $\mathcal K$ that are in $\iota(\mathcal{C})$ are represented by thicker lines.}\label{fi:prop3_2}
\end{figure}
this delivers the equality
\begin{equation}\label{eq:le_product_1}
\ov {g'}_1^{m'_1}\ov {g'}_2^{m'_2}\cdots \ov{g'}_{k_0-1}^{m'_{k_0-1}}
=
\ov c_1\cdot \ov g_{j_0}^{m}\ov g_{j_0+1}^{m_{j_0+1}}\cdots \ov g_{k_1-1}^{m_{k_1}-1}\ov g_{k_1}^{m''}\cdot \ov c_2,
\end{equation}
where $m'_1,\ldots,m'_{k_0-1}\ge M\ge M_0$, \ and $m_{j_0+1},\ldots,m_{k_1-1}\ge M\ge M_0$, $m''\le m_{k_1}$, and $|c_i|\le 2(H+2\max\{|b|,|f_0|,|f_1|\})$, $i=1,2$.

On the other hand, if $k_0=1$, we similarly (See Fig.~\ref{fi:prop3_3})
\begin{figure}[h]
 \includegraphics[width=4.5in]{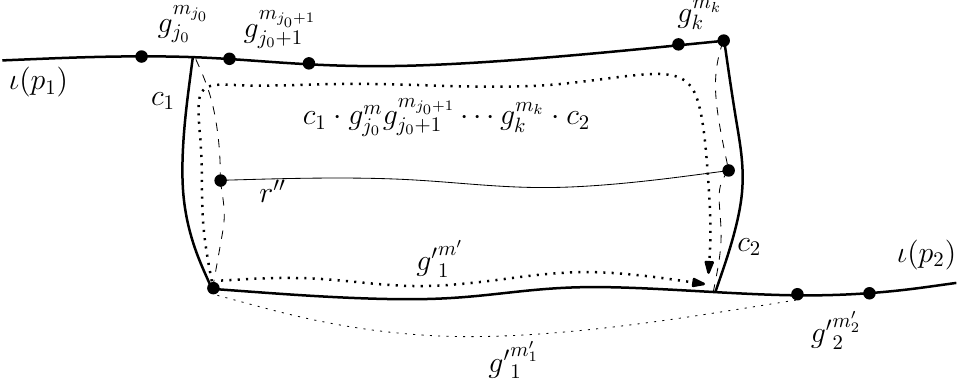}
 \caption{If $k_0=1$, the fellow travel of $r''$ with a subpath of $\iota(p_1)$ and a subpath of $\iota(p_2)$ delivers equality~\eqref{eq:le_product_2}. Paths in $\mathcal K$ that are in $\iota(\mathcal{C})$ are represented by thicker lines.}\label{fi:prop3_3}
\end{figure} obtain the equality
\begin{equation}\label{eq:le_product_2}
\ov {g'}_1^{m'}
=
\ov c_1\cdot \ov g_{j_0}^{m}\ov g_{j_0+1}^{m_{j_0+1}}\cdots \ov g_{k}^{m_k}\cdot\ov c_2,
\end{equation}
where $m'\ge M_0$ by the choice of $M$, and $m_{j_0+1},\ldots,m_{k}\ge M\ge M_0$, and $|c_i|\le 2(H+2\max\{|b|,|f_0|,|f_1|\})$, $i=1,2$.

In either case~\eqref{eq:le_product_1} or~\eqref{eq:le_product_2}, Proposition~\ref{factor lemma} applies, we get that $g'_1=g_{j_0}$ or $g'_1=g_{j_0+1}$. Note that $g'_1=b^{\pm 1}$, while $g_j=b^{\pm 1}$ only if $j=1$ or $j=k$. As we mentioned above, $j_0\le k-2$, so the only option is $j_0=0$, therefore $g'_1=g_1=b$.

Now we have that $r''$ asynchronously $(H+2\max\{|b|,|f_0|,|f_1|\})$-fellow travels with a subpath of $\iota(p_1)$ labeled by $g_{1}^{m}g_{2}^{m_{2}}\cdots g_k^{m_k}$, $m\le m_1$, and a subpath of $\iota(p_2)$ labeled by ${g'}_1^{m'_1}\cdots {g'}_{k_0-1}^{m'_{k_0-1}}{g'}_{k_0}^{m'}$. Applying the similar argument to the subpath of $\iota(p_1)$ labeled by $g_{2}^{m_{2}}\cdots g_k^{m_k}$, $m_2,\ldots,m_k\ge M$, we obtain that $g_k=b$ is equal to some $g'_j$. It cannot be $j=1$ (that would imply that $k=1$ and therefore $s_1$ is not a $b$-syllable), so $j=l$, and therefore $k_0=l$ and $g'_l=b$.

We established that $r''$ asynchronously $(H+2\max\{|b|,|f_0|,|f_1|\})$-fellow travels with a subpath of $\iota(p_1)$ labeled by $g_{1}^{m}g_{2}^{m_{2}}\cdots g_k^{m_k}$, $m\le m_1$, and subpath of $\iota(p_2)$ labeled by ${g'_1}^{m'_1}\cdots {g'_{l-1}}^{m'_{l-1}}{g'_{l}}^{m'}$, $m'\le m'_l$. Applying Proposition~\ref{factor lemma} to the corresponding labels via a similar argument again, we obtain that $(g_2,g_3,\ldots,g_k)$ is a subsequence of $(g'_1,g'_2,\ldots,g'_l)$, and that $(g'_1,\ldots,g'_{l-1})$ is a subsequence of $(g_1,\ldots,g_k)$. Taking into account that $g_1=g_k=g'_1=g'_l=b$ and none other terms are equal to $b$, it follows that $(g_1,\ldots,g_k)=(g'_1,\ldots,g'_l)$, therefore $s_1$ and $s_2$ have the same value of $b$-gap. This finishes proof of the claim.

\medskip
If the concatenation $q=q'q''$ is a reduced path, the claim above allows to assume that subpaths $q_i$ in the definition of $\widetilde{\delta_\omega}$ do not contain the common endpoint of $q',q''$, except, perhaps, $\omega=-2,-1,0,1,2$, or two other values, i.e. $\widetilde{\Delta_\omega}(q)=\widetilde{\Delta_\omega}(q')+\widetilde{\Delta_\omega}(q'')$ except for at most five values of $\omega$.

Taking possible reduction into account and applying the above claim twice, we get that $\widetilde{\Delta_\omega}(q)=\widetilde{\Delta_\omega}(q')+\widetilde{\Delta_\omega}(q'')$ except, perhaps, $\omega=-2,\ldots,2$, or four other values.
\end{proof}


\begin{lemma}\label{le:tripod} Let $D_0\in\mathbb N$ be given. Let reduced paths $q_1,q_2,q_3$ in $\mathcal K$ be such that $q_i=s_it_is_{i+1}^{-1}$,  $i\in\{1,2,3\}$ ($i+1$ is taken$\mod 3$), where $s_i,t_i$ are reduced paths such that $|t_i|<D_0$. Then if $M$ in~\eqref{Rhem} is large enough, $\widetilde{\Delta_\omega}(q_1)+\widetilde{\Delta_\omega}(q_2)+\widetilde{\Delta_\omega}(q_3)=0$ for every $\omega\in\mathbb Z$ except, perhaps, at most $72$ values.
\end{lemma}
\begin{proof}
First of all, it follows by triangle inequality and definitions of quasigeodesity and fellow travel that one can choose $M$ large enough to satisfy the following condition: if a path $q$ in $\mathcal K$ asynchronously $H$-fellow travels with a path $\iota(p)$, where $p$ is labeled by $b^M$, then the length of $q$ is at least $D_0$.

Now note that with an appropriate choice of $M$, $\widetilde{\Delta_\omega}(q_i)=\widetilde{\Delta_\omega}(s_i)+\widetilde{\Delta_\omega}(t_i)-\widetilde{\Delta_\omega}(s_{i+1})$ by Lemma~\ref{le:product}, except for at most $18$ values  of $\omega$. Then by the same lemma, $\widetilde{\Delta_\omega}(q_1q_2q_3)=\widetilde{\Delta_\omega}(q_1)+\widetilde{\Delta_\omega}(q_2)+\widetilde{\Delta_\omega}(q_3)$ except for at most $18$ values of $\omega$. Combining the equalities, we get that $\widetilde{\Delta_\omega}(q_1q_2q_3)=\widetilde{\Delta_\omega}(t_1)+\widetilde{\Delta_\omega}(t_2)+\widetilde{\Delta_\omega}(t_3)=0$ since each  $|t_i|<D_0$, except for at most $18\cdot 4=72$ values of $\omega$.
%
%
%
%
\end{proof}

\begin{prop}\label{gamma}
Suppose $w\in F$ with $d=e(w)$. Let $G$ be a non-elementary hyperbolic group. Then $\exists M\in\mathbb Z$, $\exists N\in\mathbb Z$ such that if $g\in w[G]\cap R(b,f_0,f_1,M)$, then $\gamma_d(g)\le N$.
\end{prop}
\begin{proof}
Let $w=w(x_1,\ldots,x_n)=x^{\epsilon_1}_{i_1}x^{\epsilon_2}_{i_2}\cdots x^{\epsilon_r}_{i_r}$, $\epsilon_i=\pm 1$.
Suppose $g=g_0\in G$ is such that $g_0=w(g_1,g_2,\ldots,g_n)$, $g_i\in G$. Convert this equality into a system of ``triangular'' equalities in a straightforward way:
\[\begin{array}{lcl}
g_0&=& g^{\epsilon_1}_{i_1}\cdot (g^{\epsilon_2}_{i_2}\cdots g^{\epsilon_r}_{i_r})\\
g^{\epsilon_2}_{i_2}\cdots g^{\epsilon_r}_{i_r}&=&
g^{\epsilon_2}_{i_2}\cdot( g^{\epsilon_3}_{i_3}\cdots g^{\epsilon_r}_{i_r})\\
&&\ldots\\
g^{\epsilon_{r-1}}_{i_{r-1}}g^{\epsilon_r}_{i_r}&=&
g^{\epsilon_{r-1}}_{i_{r-1}}\cdot g^{\epsilon_r}_{i_r}
\end{array}
\]
By Proposition~\ref{prop:reps}, 
there is a positive integer $D_0$ that depends only on the presentation of the group $G$ and the word $w$ such that this system admits a family of $D_0$-canonical representatives, 
i.e. there are paths $p_0,p_1,\ldots,p_n$ in $\mathcal K$, each $p_i$ originating at $\iota(1)$ and terminating at $\iota(g_i)$, such that the tripods in the graph $\mathcal K$ corresponding to the equalities above satisfy the conditions of Lemma~\ref{le:tripod}. Choosing $M$ large enough and applying this Lemma repeatedly, we obtain that \[\widetilde{\Delta_\omega}(p_0)=\epsilon_1\widetilde{\Delta_\omega}(p_{i_1})+\epsilon_2\widetilde{\Delta_\omega}(p_{i_2})+\ldots+\epsilon_r\widetilde{\Delta_\omega}(p_{i_r})
\]
for all integer $\omega$ except, perhaps, $72(r+1)$ values. Note that in the equality above, each $p_i$ appears a number of times that is a multiple of $d$ (counting inverse occurrences as $-1$), so the latter sum is $0$ mod $d$, providing
\begin{equation}\label{eq:mod-d}
\widetilde{\Delta_\omega}(p_0)=0\mod d
\end{equation}
for all positive integer $\omega$, except, perhaps, $72(r+1)$ values.

Note further that we may assume (see Proposition~\ref{prop:reps}) 
that the $D_0$-canonical representatives $p_i$ are $(\lambda',\varepsilon')$-quasigeodesic, where $\lambda',\varepsilon'$ only depend on the presentation of the group $G$ and the word $w$. Finally, observe that the element $g_0$ can be represented by a $(\lambda,\varepsilon)$-quasigeodesic word $p'$ that arises from~\eqref{Rhem}, so the paths $\iota(p')$ and $p_0$ in $\mathcal K$, corresponding to the group element $g_0$, asynchronously $H_0$-fellow travel, where $H_0$ ultimately depends only on the presentation of the group $G$ and the word $w$. Choosing $H=2H_0$ in the definition of $\widetilde{\Delta_\omega}$ (see~\eqref{eq:defn_delta_tilde} and~\eqref{eq:delta-hat}), by Lemma~\ref{le:delta_vs_tilde} we may choose $M$ large enough so that $\widetilde{\Delta_\omega}(p_0)=\Delta_\omega(p')$. Taking~\eqref{eq:mod-d} into account, we conclude $\gamma(g_0)\le {72r+72}$.
\end{proof}

\section{Proof of Theorem~\ref{width}}

For a word $w=w(x_1,\ldots,x_n)$, by $W^{(l)}\subset F$ we denote the following finite set of words in variables $x_1,\ldots, x_{nl}$:
\begin{align*}
W^{(l)}=&\{w^{\pm 1}(x_1,\ldots,x_n)w^{\pm 1}(x_{n+1},\ldots,x_{2n})\cdots\\
&\cdot w^{\pm 1}(x_{(l-1)n+1},\ldots,x_{(l-1)n+n})\}.
\end{align*}

Note that $w[G]^l=\cup_{u\in W^{(l)}}u[G]$, and $e(u)=e(w)$ for every $u\in W^{(l)}$.
For a fixed $l$, applying Proposition~\ref{gamma} to all words in $W^{(l)}$, we conclude that
for $g\in R\cap w[G]^l$, $\gamma(g)$ is bounded by a number $N$ that depends on $w$ and $l$. It is only left to note that
one can easily point out a sequence of elements $h_k$ in $R\cap  w(G)$ with $\gamma(h_k)\to\infty$.
(From now on, with an obvious abuse of notation, we don't keep track of difference between the words $b,f_0,f_1$ in $\mathcal A$ and the respective elements
$\ov b, \ov f_0,\ov f_1$ of $G$.)

{\bf Case 1.} Indeed, let $d=e(w)>0$. Denote
\[
X_j=f^M_1f^M_0 f^M_1\cdots  f^M_0 f^M_1\cdot  b^M\cdot  f^M_1 f^M_0\cdots  f_1^M f^M_0,
\]
where $f^M_0$ is repeated $j$ times on the left of $b^M$ and $j$ times on the right of it. Then put
\[
h_k=X_1^dX_2^d\cdots X_k^d.
\]
$ h_k\in  R\cap w(G)$ since it is a product of $d$-th powers. Note that $\Delta_\omega(h_k)=1$ for every
odd number between $3$ and $2k-1$ (and $\Delta_\omega(h_k)=-1 \mod d$ for every even number between $2$ and $2k$).
Therefore, $\gamma(h_k)\ge k-2\to\infty$.

{\bf Case 2.} Now, suppose $d=e(w)=0$. Renaming and/or inverting variables and passing to a conjugate,
we may assume $w(x_1,x_2,\ldots, x_n)=x_1\cdots x_2$.
Consider the following elements of the group $G$:
\[
X_{ij}= f_0^{inM+(j-1)M} f_1^M f_0^{inM+(j-1)M},\quad 1\le i <\infty,\ 1\le j\le n,
\]
and
\[
Y_j= b^{jM} f_1^Mb^{jM},\quad 1\le j\le n.
\]
Note that by Corollary~\ref{C1'}, $\{ X_{ij},Y_j\mid 1\le i <\infty,\ 1\le j\le n\} $ is an infinite free basis
for a free subgroup of $G$. Consider element $g_1\in G$ defined by
\[
g_1=w(X_{11}^{-1},X_{12},\ldots,X_{1n})=f_0^{\alpha_1}f_1^{\beta_1}f_0^{\alpha_2}\cdots f_1^{\beta_s}f_0^{\alpha_{s+1}}.
\]
Note that, since $w$ is a commutator word, $s$ is even, and therefore $s+1$ is odd, so the value
\[
\varepsilon(g_1)=\varepsilon_1+\cdots+\varepsilon_{s+1}
\]
is nonzero, where $\varepsilon_\mu=1$ if $\alpha_\mu>0$ and $\varepsilon_\mu=-1$ otherwise. Note that in any case $\alpha_1<0$
and $\alpha_{s+1}>0$. Put
\[
g_2=\left\{
      \begin{array}{ll}
        w(X_{21}^{-1},X_{22},\ldots,X_{2n})g_1, & \mathrm{if}\ \varepsilon(g_1)>0; \\
        g_1w(X_{21}^{-1},X_{22},\ldots,X_{2n}), & \mathrm{if}\ \varepsilon(g_1)<0.
      \end{array}
    \right.
\]
Note that $\varepsilon(w(X_{21}^{-1},X_{22},\ldots,X_{2n}))=\varepsilon(g_1)$, so $|\varepsilon(g_2)|\ge 2 |\varepsilon(g_1)|>|\varepsilon(g_1)|$. Proceed in the same fashion: if $i\ge 3$ and
\[
g_{i-1}=f_0^{\alpha'_1}f_1^{\beta'_1}f_0^{\alpha'_2}\cdots f_1^{\beta'_{s'}}f_0^{\alpha'_{s'+1}}
\]
with $\alpha'_1<0$, $\alpha'_{s'+1}>0$, we put
\[
g_{i}=\left\{
      \begin{array}{ll}
        w(X_{i1}^{-1},X_{i2},\ldots,X_{in})g_{i-1}, & \mathrm{if}\ \varepsilon(g_{i-1})>0; \\
        g_{i-1}w(X_{i1}^{-1},X_{i2},\ldots,X_{in}), & \mathrm{if}\ \varepsilon(g_{i-1})<0.
      \end{array}
    \right.
\]
Then the initial power of $f_0$ in $g_{i}$ is negative, the terminal one is positive, and
$\varepsilon(w(X_{i1}^{-1},X_{i2},\ldots,X_{in}))=\varepsilon(g_1)$, so $|\varepsilon(g_{i})|>|\varepsilon(g_{i-1})|$.

Finally, denote $w(Y_1,\ldots,Y_n)=B$ and consider elements of $G$ defined by
\[
h_k=Bg_1Bg_2B\ldots Bg_k B\in  w(G)\cap R.
\]
Note that $\Delta_{\varepsilon(g_i)}(h_k)=1$ for any $1\le i\le k$. Therefore, $\gamma(h_k)\ge k\to \infty$.

This finishes the proof of Theorem~\ref{width}.

\section{Appendix}\label{se:async}
Here we prove Lemma~\ref{fellow-equiv}.

\noindent \textbf{Lemma~\ref{fellow-equiv}}. \textit{Suppose $p, q$ are $K$-close $(\lambda,\varepsilon)$-quasigeodesics in a geodesic metric space $\mathcal H$ that originate from points $p_0,q_0$, respectively, such that $|p_0-q_0|\le K$, and terminate at points $p_1,q_1$, respectively, such that $|p_1-q_1|\le K$.
Then $p$ and $q$ asynchronously $K'$-fellow travel for some constant $K'=K'(\lambda, \varepsilon, K)$.}
\begin{proof}
Note that for given $\lambda,\varepsilon$ there is a function $D(r):\mathbb R_{\ge 0}\to \mathbb R_{\ge 0}$ with the following property: if $x_0\in\mathcal H$ and two points $\gamma(t'),\gamma(t'')$ of a $(\lambda,\varepsilon)$-quasigeodesic $\gamma$ in $\mathcal H$ are such that $|x_0-\gamma(t')|\le R$, $|x_0-\gamma(t'')|\le R$, then $|x_0-\gamma(t)|\le D(R)$ for every $t$ between $t'$, $t''$. Indeed, by quasigeodesity we have $|\gamma(t)-\gamma(t')|\le \lambda|t-t'|+\varepsilon\le \lambda|t'-t''|+\varepsilon\le \lambda\cdot \lambda(|\gamma(t')-\gamma(t'')|+\varepsilon)+\varepsilon\le 
\lambda^2(2R+\varepsilon)+\varepsilon,$ therefore $|x_0-\gamma(t)|\le |x_0-\gamma(t')|+|\gamma(t')-\gamma(t)|\le D(R)$, where $D(R)=R+\lambda^2(2R+\varepsilon)+\varepsilon$.

Assume $p:[0,T]\to \mathcal H$ and $q:[0,S]\to\mathcal H$. We will show that $p,q$ asynchronously $K'$-fellow travel with $K'=D(D(K+\lambda+\varepsilon))+\lambda+\varepsilon$ by finding a ``correspondence'' subset $\mathcal F\subseteq [0,T]\times [0,S]$ such that
\begin{enumerate}
\item[(a)] if $(t,s)\in\mathcal F$, then $|p(t)-q(s)|\le K'$, and
\item[(b)] if $(t,s),(t',s')\in\mathcal F$, then $t\le t'\iff s\le s'$,
\item[(c)] the projection of $\mathcal F$ onto the first coordinate is $[0,T]$, and the projection of $\mathcal F$ onto the second coordinate is $[0,S]$.
\end{enumerate}
Choose $0=t_0<t_1<\ldots<t_m=T$ in $[0,T]$ so that $|t_j-t_{j+1}|\le 1$. By $K$-closeness, for each $t_j$ there is a point $s_j\in [0,S]$ such that $|p(t_j)-q(s_j)|\le K$. Since $|p(0)-q(0)|, |p(T)-q(S)|\le K$, we may assume $s_0=0$ and $s_m=S$. We will alter the sequence $s_0,s_1,\ldots, s_m$, obtaining a {\em monotone} sequence $0=u_0\le u_1\le \ldots\le u_m=S$ as follows. Starting with $i=0$, we consequently check whether $s_{i+1}\ge s_i$.
\begin{itemize}
\item[\em Up] If the answer is positive, we set $u_{i+1}=s_{i+1}$, and move on to the next value of $i$. Note that in such case
\begin{equation}\label{eq:tkuk1}
|p(t_{i+1})-q(u_{i+1})|\le K.
\end{equation}
\item[\em Down] (Fig.~\ref{fi:down}) If the answer is negative, i.e. $s_{i+1}<s_i$, let $j$ be the least index $j>i$ such that $s_j\ge s_i$ (such index exists because $s_m=S\ge s_i$). Set $u_{i+1}=u_{i+2}=\ldots=u_{j-1}=s_i$, $u_j=s_j$, and move the process on to $i=j$. In this case, note the following.\\
-- The distance $|p(t_j)-q(s_{j-1})|\le |p(t_j)-p(t_{j-1})|+|p(t_{j-1})-q(s_{j-1})|\le \lambda\cdot |t_j-t_{j-1}|+\varepsilon +K\le \lambda+\varepsilon+K$. Put $x_0=p(t_j)$. We have $|x_0-q(s_{j-1})|,|x_0-q(s_j)|\le \lambda+\varepsilon+K$, so by the choice of the function $D$, the distance $|p(t_j)-q(u_{j-1})|=|p(t_j)-q(s_i)|\le D(\lambda+\varepsilon +K)$, since $s_i\in[s_{j-1},s_j]$ by the choice of~$j$.\\
-- Now put $x_0=q(s_i)$. By above, we have
$|p(t_i)-x_0|\le K\le D(\lambda+\varepsilon +K)$, and
\begin{equation}\label{eq:vert_ends}
|p(t_j)-x_0|=|p(t_j)-q(u_{j-1})|\le D(\lambda+\varepsilon +K),
\end{equation}
so by the choice of the function $D$,
\begin{equation}\label{eq:horizontal}
|p(t_k)-q(u_k)|\le D(D(\lambda+\varepsilon +K))\quad \mbox{for all }k=i+1,\ldots,j-1,
\end{equation}
and
\begin{equation}\label{eq:tkuk2}
|p(t_j)-q(u_j)|\le K.
\end{equation}
\begin{figure}[h]
 \includegraphics[height=1.6in]{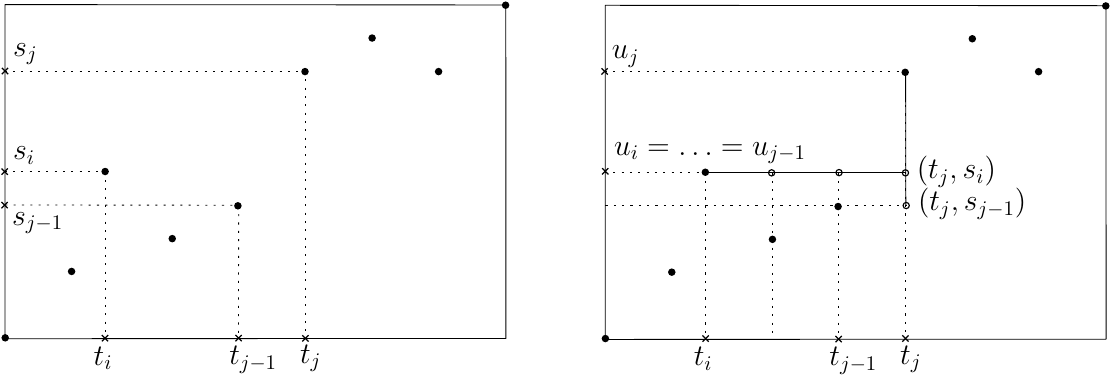}
 \caption{Case {\em Down} of constructing the monotone sequence $(u_i)$.}\label{fi:down}
\end{figure}
\end{itemize}
We proceed in such manner until we arrive at $u_m=s_m$. Given the monotone sequence $0=u_0\le u_1\le \ldots\le u_m=S$, we define $\mathcal F$ to be the polygonal line with successive vertices
\[
(t_0,u_0),(t_1,u_0),(t_1,u_1),(t_2,u_1),\ldots, (t_{m-1},u_{m-1}), (t_m,u_{m-1}), (t_m,u_m),
\]
as shown in Fig.~\ref{fi:async}.
\begin{figure}[h]
 \includegraphics[height=1.6in]{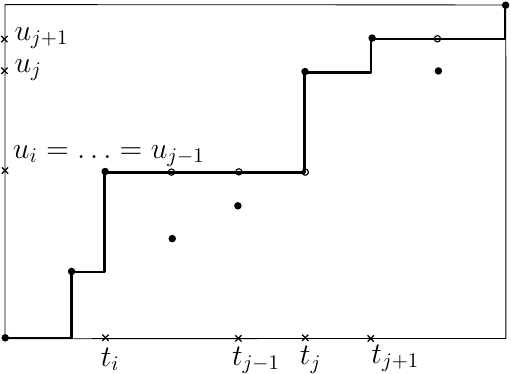}
 \caption{The set $\mathcal F$.}\label{fi:async}
\end{figure}
Since the sequence $(u_i)_{i=0,\ldots,m}$ is monotone, the condition (b) is satisfied. Let us check now the condition (a).
Combining~\eqref{eq:tkuk1},\eqref{eq:horizontal},\eqref{eq:tkuk2}, we get
\[
|p(t_k)-q(u_k)|\le D(D(\lambda+\varepsilon +K))\quad \mbox{for all } k=0,1,\ldots,m.
\]
Therefore, for each horizontal segment $[t_k,t_{k+1}]\times u_k$, we have \begin{equation}\label{eq:hor_bound}
|p(t)-q(u_k)|\le |p(t)-p(t_k)|+|p(t_k)-q(u_k)|\le \lambda+\varepsilon+D(D(\lambda+\varepsilon +K))
\end{equation} for all $t\in[t_k,t_{k+1}]$ by~\eqref{eq:horizontal}. For each vertical segment of nonzero hight $t_k\times[u_{k-1},u_k]$ we have $|p(t_k)-q(u_k)|\le K$ and, depending on whether $u_{k-1}$ comes from {\em Up} or {\em Down} case, $|p(t_k)-q(u_{k-1})|\le K+\lambda+\varepsilon$ or $|p(t_k)-q(u_{k-1})|\le D(\lambda+\varepsilon +K)$ by~\eqref{eq:vert_ends}. By the choice of the function $D$,
\begin{equation}\label{eq:vert_bound}
|p(t_k)-q(u)|\le D(D(\lambda+\varepsilon +K))
\end{equation}
for all $u\in [u_{k-1},u_k]$. Putting~\eqref{eq:hor_bound} and~\eqref{eq:vert_bound} together, we conclude that the condition (a) holds with $K'=D(D(\lambda+\varepsilon +K))+\lambda+\varepsilon$.

Finally, the condition (c) holds since $(t_0,u_0)=(0,0)$ and $(t_m,u_m)=(T,S)$.
\end{proof}

\end{document}